\documentclass[11pt]{amsart}

\usepackage[a4paper,vmargin=4cm]{geometry}
\usepackage{amsfonts,amssymb,amscd,amstext}
\usepackage{graphicx}
\usepackage{color}
\usepackage[dvips]{epsfig}

\usepackage[utf8]{inputenc}
\usepackage{hyperref}

\usepackage{verbatim}
\usepackage{fancyhdr}
\input xy
\xyoption{all}

\usepackage{times}
\usepackage{enumerate}
\usepackage{titlesec}
\usepackage{mathrsfs}

\numberwithin{equation}{section}
\numberwithin{figure}{section}

\titleformat{\section}%[display]
{\filcenter\bfseries\large} {\thesection{.}}{0.2cm}{}%[$\vspace*{-1.0cm}$]
\titleformat{\subsection}[runin]
{\bfseries} {\thesubsection{.}}{0.15cm}{}[.]
\titleformat{\subsubsection}[runin]
{\em}{\thesubsubsection{.}}{0.15cm}{}[.]
\usepackage[up,bf]{caption}
\newtheorem{theorem}{Theorem}[section]
\newtheorem{proposition}[theorem]{Proposition}

\newtheorem{lemma}[theorem]{Lemma}
\newtheorem{corollary}[theorem]{Corollary}

\theoremstyle{definition}
\newtheorem{definition}[theorem]{Definition}
\newtheorem{remark}[theorem]{Remark}

\newtheorem{problem}[theorem]{Problem}
\newtheorem{example}[theorem]{Example}

\newcommand\Ecal{\mathcal{E}}
\newcommand\Hcal{\mathcal{H}}

\newcommand\Cscr{\mathscr{C}}

\newcommand\Oscr{\mathscr{O}}

\newcommand\B{\mathbb{B}}
\newcommand\C{\mathbb{C}}

\newcommand\CP{\mathbb{CP}}

\newcommand\N{\mathbb{N}}

\newcommand\R{\mathbb{R}}

\newcommand\ggot{\mathfrak{g}}

\newcommand\igot{\mathfrak{i}}

\renewcommand\igot{\mathfrak{i}}

\newcommand\E{\mathrm{e}}

\renewcommand\imath{\igot}

\newcommand\hra{\hookrightarrow}

\newcommand\di{\partial}

\renewcommand\span{\mathrm{span}}

\newcommand\Id{\mathrm{Id}}

\newcommand\Aut{\mathrm{Aut}}

\def\span{\mathrm{span}}

\def\Ell1{\mathrm{Ell_1}}
\def\CEll1{\mathrm{CEll_1}}

\def\rmax{\mathrm{rmax}}

\begin{document}

\fancyhead[LO]{Oka domains in Euclidean spaces}
\fancyhead[RE]{F.\ Forstneri\v c, E. F. \ Wold} 
\fancyhead[RO,LE]{\thepage}

\thispagestyle{empty}

%% Title
%\vspace*{5mm}

\begin{center}
{\bf \LARGE Oka domains in Euclidean spaces}

\vspace*{0.5cm}

%% Authors
{\large\bf  Franc Forstneri{\v c} and Erlend Forn\ae ss Wold} 
\end{center}

\vspace*{0.5cm}

{\small
\noindent {\bf Abstract}\hspace*{0.1cm}
In this paper we find surprisingly small Oka domains in Euclidean spaces $\C^n$ of dimension $n>1$ at the very limit of 
what is possible. Under a mild geometric assumption on a closed unbounded convex set $E$ in $\C^n$ 
we show that $\C^n\setminus E$ is an Oka domain. In particular, there are Oka domains which are only slightly bigger 
than a halfspace, the latter being neither Oka nor hyperbolic. 
This gives smooth families of real hypersurfaces $\Sigma_t\subset \C^n$ for $t\in\R$ dividing $\C^n$ in an unbounded 
hyperbolic domain and an Oka domain such that at the threshold value $t=0$ the hypersurface 
$\Sigma_0$ is a hyperplane and the character of the two sides gets reversed.  
More generally, we show that if $E$ is a closed set in $\C^n$ for $n>1$ whose projective closure $\overline E\subset\CP^n$ 
avoids a hyperplane $\Lambda\subset\CP^n$ and is polynomially convex in 
$\CP^n\setminus \Lambda\cong\C^n$, then $\C^n\setminus E$ is an Oka domain.
}

\noindent{\bf Keywords}\hspace*{0.1cm} 
Oka manifold, hyperbolic manifold, density property, projectively convex set
\vspace*{0.1cm}

\noindent{\bf MSC (2010):}\hspace*{0.1cm}  32Q56, 32E30, 32M17
%
%  FF: these are good classification numbers
%
%  32B15: (1973-now) Analytic subsets of affine space
%  32H05: Holomorphic mappings, embeddings,...
%
%  32E10: Stein manifolds
%  32E30: (1973-now) Holomorphic and polynomial approximation, Runge pairs, interpolation
%  32H02: (1991-now) Holomorphic mappings, (holomorphic) embeddings and related questions
%  32M17: Automorphism groups of Cn and affine manifolds 
%  32Q56: Oka principle and Oka manifolds 
%

\noindent {\bf Date: \rm 24 March 2022}

%%%%%%%%%%
%%%%%%%%%%
%%%%%%%%%%
%%%%%%%%%%
%%%%%%%%%%
%%%%%%%%%%

%
%   INTRODUCTION 
%
\section{Introduction}\label{sec:intro}  
A complex manifold $Y$ is said to be an Oka manifold if it admits many holomorphic maps $X\to Y$ from 
any Stein manifold $X$  (see \cite[Definition 5.4.1 and Theorem 5.4.4]{Forstneric2017E} and
\cite{Larusson2010NAMS}). In particular, every continuous map $X\to Y$ must be homotopic to a holomorphic map,
with Runge approximation on a compact holomorphically convex subset and interpolation on 
a closed complex subvariety of $X$ where the given map happens to be holomorphic.
Oka manifolds are at the heart of many existence theorems in complex analysis, with diverse applications;
see \cite{AlarconForstnericLopez2021,Forstneric2017E,Forstneric2022Oka,ForstnericLarusson2011}. 
Oka manifolds are at the opposite end of the spectrum from hyperbolic manifolds 
which do not admit any nonconstant holomorphic images of $\C$ (see Kobayashi \cite{Kobayashi1970}).

Most complex manifolds are neither hyperbolic nor Oka, but have a mixture of both properties. 
For example, a halfspace $\{(z',z_n)\in\C^n: \Im z_n>0\}$ is the product of a halfplane, which is hyperbolic,
and the affine space $\C^{n-1}$, which is Oka. Hence, no domain contained in a halfspace is Oka.
In this paper we show that complements of most closed unbounded convex sets in $\C^n$ for $n>1$ are Oka; 
see Theorems \ref{th:strictlyconvex}, \ref{th:convex}, and \ref{th:convexnoline}. In particular, there are Oka domains 
which are only slightly bigger than a halfspace. This represents a major advancement in the theory of Oka manifolds. 

Let $E$ be a closed domain in $\C^n$ with $\Cscr^1$ boundary.
Given a point $p\in bE$, we denote by $T_p bE$ the affine tangent hyperplane to $E$ at $p$ and by 
$T^\C_p bE$ the unique affine complex hyperplane in $T_p bE$ passing though $p$. 
If $E$ is convex then $E\cap T_p bE \subset bE$. One of our main results is the following; 
it is proved in Section \ref{sec:convex}. Another main result is Theorem \ref{th:main}.

%
%   MAIN THEOREM ON CONVEX SETS
%
\begin{theorem}\label{th:strictlyconvex}
If $E$ is a closed convex set with $\Cscr^1$ boundary in $\C^n$ for $n>1$ such that  
$E\cap T^\C_p bE$ does not contain an affine real halfline for any $p\in bE$, then $\C^n\setminus E$ is an Oka domain. 
\end{theorem}

Theorem \ref{th:strictlyconvex} is new for unbounded convex sets. For compact sets $E$ in $\C^n$
with $n>1$ it is known since 2020 that $\C^n\setminus E$ is Oka provided that $E$ is polynomially convex,
which includes all convex sets
(see Kusakabe \cite[Theorem 1.2 and Corollary 1.3]{Kusakabe2020complements} and \cite{ForstnericWold2020MRL}). 
These were the first examples of Oka domains in $\C^n$ whose complements have nonempty interior.
The only previously known examples of such domains were complements of hyperplanes and of tame (in particular, of algebraic)
complex subvarieties of codimension at least two in $\C^n$. 

A closed convex set $E$ in $\R^n$ is said to be {\em strictly convex} if the interior of the line segment between 
any pair of points in $E$ is contained in the interior of $E$. % \cite[Def.\ 7.6]{Valentine1964}. 
Equivalently, the boundary of $E$ does not contain any line segment. 
%The following is a corollary to Theorem \ref{th:strictlyconvex}.   

%
%   STRICTLY LINEARLY CONCAVE DOMAINS ARE OKA
%
\begin{corollary}\label{cor:strictlyconvex} 
If $E$ is a closed strictly convex domain with $\Cscr^1$ boundary in $\C^n$ for $n>1$, 
then its complement $\C^n\setminus E$ is an Oka domain. 
\end{corollary}

There are many examples satisfying Theorem \ref{th:strictlyconvex} which are of the form
\begin{equation}\label{eq:graph}
	E=\{(z',z_n)\in \C^n: \Im z_n \ge \phi(z',\Re z_n)\},
\end{equation}
where $\phi$ is a convex function of class $\Cscr^1$. (Here, $\Re$ and $\Im$ denote, respectively,
the real and the imaginary part.) A simple example is the closure of the Siegel upper halfspace 
\begin{equation}\label{eq:Siegel}
	\bigl\{z=(z',z_n)\in \C^n: \Im z_n > |z'|^2\bigr\}.
\end{equation}
This domain is biholomorphic to the ball $\{w\in \C^n:|w|<1\}$ via the Cayley transform
\begin{equation}\label{eq:Cayley}
	z=\Phi(w',w_n) = \imath \left(  \frac{w'}{1-w_n},\frac{1+w_n}{1-w_n}\right)\!.
\end{equation}
(See Rudin \cite[Sec.\ 2.3]{Rudin2008}.)
Its boundary $\{\Im z_n = |z'|^2\}$ is strongly convex in the $z'$ direction and 
is foliated by translates of the $\Re z_n$ axis. Theorem \ref{th:strictlyconvex} 
shows that the Siegel lower halfspace $\{(z',z_n)\in \C^n: \Im z_n < |z'|^2\bigr\}$ is an Oka domain.

Theorem \ref{th:strictlyconvex} and Corollary \ref{cor:strictlyconvex} imply the following interesting phenomenon.
Assume that $E$ is a stricty convex set of the form \eqref{eq:graph}.
Equivalently, $\phi$ is a strictly convex function, meaning that 
for every pair of distinct points $a,b\in \C^{n-1}\times \R$ we have that
\[
	\phi(ta+(1-t)b) < t\phi(a)+(1-t) \phi(b)\ \ \text{for all $0<t<1$}.
\]
Consider the $1$-parameter family of real hyperplanes
\[
	\Sigma_t=\{(z',z_n)\in \C^n: \Im z_n = t \phi(z',\Re z_n)\}\quad \text{for}\ t\in\R.
\]
If $t>0$, the convex domain $\Omega_t^+=\{\Im z_n > t \phi(z',\Re z_n)\}$ above $\Sigma_t$ does not contain 
any affine complex line, so it is hyperbolic (see \cite{Barth1980,BracciSaracco2009}), %,Harris1979}),
while the domain $\Omega_t^-=\{\Im z_n < t \phi(z',\Re z_n)\}$ below $\Sigma_t$ is Oka by Theorem 
\ref{th:strictlyconvex}. For $t<0$ the picture is reversed, while at $t=0$ the hyperplane $\Sigma_0=\{\Im z_n=0\}$ 
splits $\C^n$ in a pair of halfspaces. The same conclusion holds if we rescale the Siegel domain \eqref{eq:Siegel},
or any domain of the form $\{\Im z_n\ge \phi(z')\}$ where $\phi$ is strictly convex.

These are the first known examples of splitting $\C^n$ by a smooth family of hypersurfaces 
into pairs of an unbounded hyperbolic domain and a (necessarily unbounded) Oka domain 
such that the nature of the two domains gets reversed at some value of the parameter. 

There are examples in the literature of holomorphic families of compact Oka manifolds degenerating to 
a non-Oka manifold; see \cite[Corollary 5]{ForstnericLarusson2014IMRN}. A recent example with open manifolds 
(see \cite[Theorem 10.1]{Forstneric2022Oka}) gives a holomorphic fibration $X\to \C$, with
$X$ is a Stein domain in $\C^3$, that is trivial over $\C^*=\C\setminus\{0\}$ with fibres being 
Fatou--Bieberbach domains in $\C^2$, which degenerate over $0$ to the product of a disc with $\C$.
However, the reversal of the nature of the two sides, observed above, does not 
occur in this example. We also mention that there are Fatou--Bierberbach domains with 
hyperbolic complements. Indeed, every closed complex submanifold of $\C^n$ is contained in a Fatou--Bierberbach domain 
(see Wold \cite{Wold2005}) and there are proper holomorphic embeddings $\C^{n-1}\hra\C^n$ with 
hyperbolic complements (see Buzzard and Forn\ae ss \cite{BuzzardFornaess1996}
and Borell and Kutzschebauch \cite{BorellKutzschebauch2006}).
However, it is not known whether there is a Fatou--Bierberbach domain whose closure has Oka complement.

%
%	 WHAT IS KNOWN
%
Earlier examples of Oka domains with big complements in $\C^n$ for $n\ge 3$ were found
by Kusakabe \cite[Theorem 1.6]{Kusakabe2020complements}. He showed that for any 
closed polynomially convex set $E$ contained in a set $\{(z',z'')\in \C^{n-2}\times \C^2:|z''|\le c(1+|z'|)\}$ 
for some $c>0$, the complement $\C^n\setminus E$ is Oka. Nevertheless, Oka domains of this type are much 
bigger than some of those given by Theorem \ref{th:strictlyconvex}, especially in low dimensions.
There are also examples of compact non-polynomially convex sets in $\C^n$ for $n>1$ 
with Oka complements; see \cite[Theorem 4.10]{Forstneric2022Oka}. In particular, the complement of 
a compact rectifiable Jordan curve in $\C^n$ for $n>1$ is Oka.

%
%	 MORE GENERAL CONVEX SETS WITH OKA COMPLEMENTS
%

We now describe more general closed unbounded convex sets in $\C^n$ with Oka complements. These results use 
Theorem \ref{th:strictlyconvex} or Corollary \ref{cor:strictlyconvex}, 
together with approximation of closed convex sets from the outside by more regular ones.
%
%
%   MORE GENERAL CONVEX GRAPHS
%
%
We shall say that a convex function $\phi:\R^n\to\R$ is {\em irreducible} if it is not of the form $\phi=\psi\circ P+l$ 
where $P:\R^n\to \R^m$ is a linear projection with $m<n$, $\psi$ is a convex function on $\R^m$, and $l$ is a linear 
function on $\R^n$. This means that $\phi$ is not a convex function of a smaller number of variables 
which is linear in the remaining variables.
%By Corollary \ref{cor:strictlyconvex}, together with approximation theorems for convex functions due to Azagra \cite{Azagra2013} and the fact that an increasing union of Oka manifolds is Oka (see \cite[Proposition 5.6.7]{Forstneric2017E}) we obtain the following result for more general convex sets of the form \eqref{eq:graph}.

\begin{theorem}\label{th:convex}
If $\phi$ is an irreducible convex function on $\C^{n-1}\times\R$, then the domain
\[
	\Omega_\phi =\{(z',z_n)\in\C^n : \Im z_n<\phi(z',\Re z_n)\}
\]
 is Oka. The same is true for domains of the form
\[
	\Omega_\phi =\{(z',z_n)\in\C^n : \Im z_n<\phi(z')\} 
\]
where $\phi:\C^{n-1} \to \R$ is an irreducible convex function. 
\end{theorem}

\begin{proof}
By Azagra \cite[Theorem 1.1 and Proposition 1.6]{Azagra2013} the condition that $\phi$ is irreducible implies that  
for every $\epsilon>0$ there is a smooth strictly convex function $\psi:\C^{n-1}\times\R\to\R$ such that
$\phi-\epsilon<\psi <\phi$. Hence, the domain $\Omega_{\psi}=\{\Im z_n<\psi(z',\Re z_n)\}$
is Oka by Corollary \ref{cor:strictlyconvex}. This gives an increasing sequence
$\phi_1<\phi_2<\phi_3<\cdots$ of smooth strictly convex functions on $\C^{n-1}\times\R$ 
converging uniformly to $\phi$ such that the sequence of Oka domains $\Omega_{\phi_j}$ 
increases to the domain $\Omega_\phi$ as $j\to\infty$. By \cite[Proposition 5.6.7]{Forstneric2017E} it follows 
that $\Omega_\phi$ is Oka. A similar argument holds in the second case, where the new domain 
$\Omega_{\psi}=\{\Im z_n<\psi(z')\}$ is Oka by Theorem \ref{th:strictlyconvex} since the 
real lines contained in $b\Omega_{\psi}=\{\Im z_n=\psi(z')\}$ 
(in the $\Re z_n$ direction) are not complex tangent to the boundary. 
\end{proof}

Let us illustrate Theorem \ref{th:convex} by an example.

%
%   EXAMPLE: A CONVEX CONE
%
\begin{example}\label{ex:cone}
Every open set in $\C^n$ of the form
\[
	\Im z_n <  c |\Re z_n| + \sum_{j=1}^{n-1} \bigl(a_j |\Re z_j|+ b_j|\Im z_j|\bigr) 
\]
for $c\ge 0$ and strictly positive numbers $a_1,\ldots,a_{n-1}, b_1,\ldots, b_{n-1}$ is an Oka domain. 
\end{example}

%
%  CONVEX SETS WITHOUT A REAL LINE 
%

%Finally, we obtain the following simple geometric condition on a closed convex set for its complement to be Oka.

We also have the following result. 

\begin{theorem}\label{th:convexnoline}
If $E$ is a closed convex set in $\C^n$ for $n>1$ which does not contain any affine real line, 
then $\mathbb C^n\setminus E$ is an Oka domain. 
\end{theorem}

\begin{proof}
By Theorem \ref{thm:int} there is a nested sequence $E_1\supset E_2\supset E_3\supset\cdots$  of smoothly 
bounded strictly convex sets in $\C^n$ such that $E=\bigcap_{j=1}^\infty E_j$. By Corollary \ref{cor:strictlyconvex}
the domain $\Omega_j=\mathbb C^n\setminus E_j$ is Oka for every $j\in\N$. Hence, 
$\C^n\setminus E=\bigcup_{j=1}^\infty\Omega_j$ is the increasing union of domains $\Omega_j$, 
so it is Oka by \cite[Proposition 5.6.7]{Forstneric2017E}.
\end{proof}

%
%   DISCUSSION
%
The results presented so far show that complements of most closed convex sets in $\C^n$ for $n>1$ are Oka. 
They also give a partial answer to \cite[Problem 4.13]{Forstneric2022Oka}, 
asking whether it is possible to characterise Oka domains in $\C^n$ in terms of geometric properties 
of their boundaries, in analogy to the classical Levi problem characterizing smoothly bounded domains 
of holomorphy as the Levi pseudoconvex ones. Since the biholomorphically invariant version of 
strong convexity is strong pseudoconvexity, it is natural to ask the following.

%
%  PROBLEM: ARE STRONGLY LEVI PSEUDOCONCAVE DOMAINS OKA?
%
\begin{problem}\label{prob:pseudoconcave}
\begin{enumerate}[\rm (a)]
\item Is every domain with connected strongly Levi pseudoconcave boundary in $\C^n$ for $n>1$ an Oka domain?
\item Is every smoothly bounded Oka domain in $\C^n$ Levi pseudoconcave?
\end{enumerate}
\end{problem}

Part (a) does not follow from Theorem \ref{th:strictlyconvex} since the latter assumes (strict) concavity 
in a global holomorphic coordinate system. Note that an Oka domain cannot have any local peak points for 
plurisubharmonic functions as this would yield a nonconstant bounded plurisubharmonic function on the domain. 
In particular, an Oka domain has no strongly pseudoconvex boundary points.

Part (b) of the above problem is a kind of {\em inverse Levi problem}. It has been known since Oka's work in 1940s
(see \cite[Chaps.\ VI and IX]{Oka1984}) that a smoothly bounded domain in $\C^n$ is a domain of holomorphy
(equivalently, a Stein domain) if and only if its boundary is Levi pseudoconvex.
% , meaning that its Levi form is nonnegative in the complex tangential directions. 
Oka manifolds are in many ways dual to Stein manifolds, a fact made precise by 
L\'arusson's model category for Oka theory (see \cite{Larusson2004} and \cite[Sect.\ 7.5]{Forstneric2017E})
in which Oka manifolds are fibrant and Stein manifolds are cofibrant. It is therefore natural to expect that 
these two classes of domains in $\C^n$ are also dual to each other in the geometric sense. 
In particular, it may transpire that a smoothly bounded domain in $\C^n$ for $n>1$ which is both Oka and Stein
has Levi flat boundary. If true, this would be a truly interesting new paradigm in complex analysis.

Here is another open problem.

\begin{problem}
Is there a smooth real hypersurface $\Sigma$ in $\C^n$ for $n>1$ whose complement $\C^n\setminus \Sigma$
is a union of Oka domains? The same question for $\CP^n$.
\end{problem}

In dimension $n=2$, a smooth hypersurface splitting $\C^2$ or $\CP^2$ into a union of Oka domains is necessarily
Levi-flat since the boundary of an Oka domain cannot contain a Levi strongly 
pseudoconvex point. The question on the existence of a smooth Levi-flat hypersurface in $\CP^2$ 
is a well-known and long-standing open problem, and the answer is known to be negative in $\CP^n$
for $n>2$; see Siu \cite{Siu2000}. On the other hand, Stens{\o}nes \cite{Stensones1997}
constructed Fatou--Bieberbach domains in $\C^n$ for any $n>1$ having smooth boundaries.
It is not known whether there is such a domain whose closure has Oka complement.

We now present our main result from which Theorem \ref{th:strictlyconvex} and other results in the paper are derived.
We consider $\C^n$ as an affine domain in the projective space $\CP^n=\C^n\cup H$, 
where $H=\CP^n\setminus \C^n \cong\CP^{n-1}$  is the hyperplane at infinity. 
Given a closed subset $E\subset\C^n$, we denote by $\overline E$ its topological closure in $\CP^n$. 
The main idea is to look at the closure $\overline E$ of a closed convex set as in Theorem 
\ref{th:strictlyconvex} in another affine chart on $\CP^n$ in which it becomes a
compact polynomially convex set. 

%
%   MAIN THEOREM
%
\begin{theorem}\label{th:main}
If $E$ is a closed subset of $\C^n$ for $n>1$ and $\Lambda\subset \CP^n$ is 
a projective hyperplane such that $\overline E\cap \Lambda=\varnothing$ and $\overline E$ is 
polynomially convex in $\CP^n\setminus\Lambda$, then $\C^n\setminus E$ is Oka.
\end{theorem}

\begin{remark}\label{rem:conditionE}
Choosing coordinates $z=(z',z_n)$ on $\C^n$ in which $\Lambda=\{z_n=0\}$, it is easily seen that
$\overline E\cap \Lambda=\varnothing$ if and only if the set 
$E\cap \{(z',z_n):|z_n|\le c|z'|\}$ is compact for some $c>0$.
\end{remark}

%For a set $E$ as in Theorem \ref{th:main} it was known beforehand that $\CP^n\setminus \overline E$ is Oka (see \cite[Theorem 5.1]{Forstneric2022Oka}). 

The proof of  Theorem \ref{th:main}, given in Section \ref{sec:proofmain}, combines the 
characterization of Oka manifolds by Condition $\mathrm{Ell}_1$, due to 
Kusakabe \cite{Kusakabe2021IUMJ} (see Theorem \ref{th:Ell1}), and a new
result proved in this paper concerning the existence of holomorphically
varying families of Fatou--Bieberbach domains in $(\C^{n-1}\times \C^*)\setminus K$, 
where $K$ is a polynomially convex set in $\C^n$ for $n>1$; see Theorem \ref{th:MRL2020Theorem1.1}.

%
%  SIEGEL DOMAIN AGAIN
%
\begin{example}\label{ex:Siegel}
Conditions in Theorem \ref{th:main} are easily verified when $E$ is the closure of 
the Siegel domain \eqref{eq:Siegel}. Indeed, if $z=\Phi(w)$ is the Cayley map in 
\eqref{eq:Cayley}, we have that 
\[
	\Im z_n -|z'|^2 = \frac{1-|w|^2}{|1-w_n|^2} \quad \text{and}\quad 
	w=\Phi^{-1}(z) = \left( \frac{2z'}{z_n+\imath},\frac{z_n-\imath}{z_n+\imath}\right)\!.
\]
(See Rudin \cite[Sec.\ 2.3]{Rudin2008}.)	
Hence, $\Phi$ extends to an automorphism of $\CP^n$ mapping the ball 
$\B=\{w\in\C^n: |w|<1\}$ onto the Siegel domain $E$ \eqref{eq:Siegel} so that the hyperplane 
$\{w_n=1\}$ gets mapped to the hyperplane at infinity $H=\CP^n\setminus \C^n$ in the $z$ coordinates, 
while the hyperplane $\Lambda=\{z_n=-\imath\}$ is at infinity in the $w$ coordinates. 
The closure of $E$ in $\CP^n$ in the $w$-coordinates is the closed ball $\overline \B$, which is 
polynomially convex, and $\overline E \cap \overline \Lambda=\varnothing$.
\end{example}

Theorem \ref{th:main} can be equivalently expressed as follows,
considering $H$ as the hyperplane at infinity and letting $E$ be a closed set in 
$\C^n=\CP^n\setminus H$ and $K=\overline E$ its closure in $\CP^n$.

%
%   MAIN THEOREM BIS
%
\begin{theorem}\label{th:mainbis}
Assume that $K$ is a compact subset of $\CP^n$ for $n>1$ and $\Lambda\subset \CP^n$ is 
a projective hyperplane such that $K \cap \Lambda=\varnothing$ and $K$ is 
polynomially convex in $\CP^n\setminus\Lambda\cong\C^n$.
Then, for every projective hyperplane $H\subset \CP^n$ the manifold $\CP^n\setminus (H\cup K)$ is Oka.
\end{theorem}

It is natural to look for geometric sufficient conditions on a closed set $E$ in $\C^n$ to 
satisfy Theorem \ref{th:main}. In Section \ref{sec:projconvex} we show that this holds if the topological 
closure of $E$ in $\CP^n$ is a projectively convex set, meaning that the set of projective hyperplanes 
contained in $\CP^n\setminus \overline E$ is connected and their union equals $\CP^n\setminus \overline E$   
(see Definition \ref{def:projconvex} and Theorem \ref{th:projconvex}). 
With more work involving a combination of complex, convex, and projective geometry we show 
in Section \ref{sec:convex} that every closed convex set $E$ in $\C^n$ satisfying the
conditions in Theorem \ref{th:strictlyconvex} has  projectively convex closure $\overline E\subset \CP^n$,
so Theorem \ref{th:strictlyconvex} follows from Theorem \ref{th:projconvex}.
Finally, Theorem \ref{th:convexnoline} is proved in Section \ref{sec:intersection}.

%%%%%%%%%%%%%%%%
%
%   PRELIMINARIES
%
%%%%%%%%%%%%%%%%
\section{Holomorphic families of Fatou--Bieberbach domains avoiding a hyperplane and a polynomially convex set}\label{sec:prelim}

In this section we develop the relevant tools % from Anders\'en--Lempert theory 
which are used in the proof of Theorem \ref{th:main}. The main result of the section is 
Theorem \ref{th:MRL2020Theorem1.1}; see also Corollary \ref{cor:Oka}.
It gives holomorphic families of Fatou--Bieberbach domains in 
$\C^{n-1}\times\C^* \setminus K$, where $K$ is a compact polynomially convex set in 
$\C^n$ for some $n>1$. Hence, these domains avoid both a complex hyperplane and a 
polynomially convex set, so they are fairly small. Finding small Fatou--Bieberbach domains
is of interest also in connection to the still open Michael's problem; see 
Dixon and Esterle \cite{DixonEsterle1986}. 

Recall that a Lie algebra $\ggot$ of holomorphic vector fields on a complex manifold $X$ is said to have
the {\em density property} if the Lie subalgebra $\ggot_0$ of $\ggot$, generated by all $\C$-complete vector fields 
in $\ggot$ (using sums and Lie brackets), is dense in $\ggot$ in the compact-open topology
(see Varolin \cite{Varolin2001} or \cite[Sect.\ 4.10]{Forstneric2017E}).
If $X$ is an algebraic manifold and $\ggot$ consists of algebraic vector fields, then $\ggot$
has the {\em algebraic density property} if $\ggot_0=\ggot$.

We recall the following result due to Varolin  \cite[Theorem 5.1 (1)]{Varolin2001}.

%
%   VAROLIN
%
\begin{theorem}\label{th:density}
If $1\le k<n$ then the Lie algebra $\ggot^{n,k}$ of holomorphic vector fields on $\C^n=\C^k\times\C^{n-k}$ 
that vanish on $\C^k\times\{0\}^{n-k}$ has the density property, and the Lie algebra of polynomial vector
fields with the same property has the algebraic density property.
This holds in particular for the Lie algebra of holomorphic vector fields vanishing on 
the hyperplane $\{z_n=0\}=\C^{n-1}\times \{0\}$.
\end{theorem}

The algebraic case of the above result is not explicitly stated in \cite{Varolin2001}, 
but is evident from \cite[proof of Theorem 5.1]{Varolin2001}.
The holomorphic case will suffice for our needs.

The following corollary to Theorem \ref{th:density} is seen by following 
\cite[proof of Theorem 4.9.2]{Forstneric2017E} (originally proved in \cite{ForstnericRosay1993}). 
This is the key argument of the Anders\'en--Lempert approximation theory for isotopies of injective
holomorphic maps by holomorphic automorphisms; cf.\ \cite{AndersenLempert1992}. 
See also \cite[Theorem 2.5]{Varolin2001} and \cite[Theorem 2.12]{ForstnericKutzschebauch2021AM}. 

%
%   APPROXIMATION WITH INTERPOLATION
%
\begin{theorem}\label{th:interpolationY}
Assume that $\Omega$ is a Stein Runge domain in $\C^n$ for $n>1$ 
and $\Phi_t:\Omega\to \C^n$ $(t\in [0,1])$ is an isotopy of injective holomorphic maps such that 
$\Phi_0$ is the identity map on $\Omega$, and for every $t\in [0,1]$ the domain $\Phi_t(\Omega)$ 
is Runge in $\C^n$ and $\Phi_t$ agrees with the identity map on $\{z_n=0\}\cap \Omega$.
Then $\Phi_1$ can be approximated uniformly on compacts in
$\Omega$ by holomorphic automorphisms of $\C^n$ fixing $\{z_n=0\}$ pointwise.
\end{theorem}

%
%   PARAMETERIC VERSION
%
\begin{remark}\label{rem:parametric}
We shall  need a parametric version of this result in which $\Omega$ is a Stein Runge domain 
in $\C^N\times\C^n$ with coordinates $\zeta\in\C^N$ and $z\in\C^n$ for some $N\in\N$, 
and we are considering isotopies of injective holomorphic maps of the form
\begin{equation}\label{eq:Phi}
	\Phi(\zeta,z) = (\zeta,\phi(\zeta,z))\ \ \text{for}\ \ (\zeta,z) \in\Omega.  
\end{equation}
The basic case without interpolation on $\{z_n=0\}\cap \Omega$ is \cite[Theorem 4.9.10]{Forstneric2017E}.
To prove the parametric version of the density property in Theorem \ref{th:density},
one works with complete holomorphic vector fields whose coefficients are holomorphic functions
of the parameter $\zeta$ (see Kutzschebauch \cite{Kutzschebauch2005}).
This in turn implies the parametric version of Theorem \ref{th:interpolationY}.
\end{remark}

By using the parametric version of Theorem \ref{th:interpolationY} we now prove the following result,
which is the main analytic ingredient in the proof of Theorem \ref{th:main}.

%
%   THEOREM
%
\begin{theorem}\label{th:MRL2020Theorem1.1} 
Assume that $K$ is a compact polynomially convex set in $\C^n$ for some $n>1$, 
$L$ is a compact polynomially convex set in $\C^N$ for some $N\in\N$, 
and $f:U\to \C^n$ is a holomorphic map on an open neighbourhood $U\subset \C^N$ of $L$ such that 
\[
	f(\zeta)\in (\C^{n-1}\times\C^*) \setminus K\ \ \text{holds for all $\zeta \in L$}. 
\]
Then there are a neighbourhood $V\subset U$ of $L$ and a holomorphic map $F:V\times \C^n\to \C^n$ 
such that for every $\zeta \in V$ we have that 
\[
	\text{$F(\zeta,0)=f(\zeta)$ and the map $F(\zeta,\cdotp):\C^n\to (\C^{n-1}\times\C^*) \setminus K$ is injective.}
\]
\end{theorem}

It follows that 
\[
	\Omega_\zeta = \{F(\zeta,z):z\in\C^n\} \subset  (\C^{n-1}\times\C^*) \setminus K
\]
is a family of Fatou--Bieberbach domains depending holomorphically on the parameter $\zeta \in V$.
This result is similar in spirit to \cite[Theorem 1.1]{ForstnericWold2020MRL}, but 
the Fatou--Bieberbach domains which we construct here also avoid the hyperplane $\{z_n=0\}$.
This addition is crucial for applications in this paper.

\begin{proof}
Since the set $L$ is polynomially convex, we may assume that $U$ is Stein and Runge in $\C^N$. 
Then, $X=U\times \C^n$ is a Runge Stein domain in $\C^{N+n}$ and the graph 
$\Gamma=\{(\zeta,f(\zeta))\in X: \zeta\in U\}$ is a closed Stein submanifold of $X$. The restricted graph 
\begin{equation}\label{eq:GammaL}
	\Gamma_L=\{(\zeta,f(\zeta))\in X: \zeta\in L\} \subset L\times \left((\C^{n-1}\times \C^*) \setminus K\right) 
\end{equation}
is clearly $\Oscr(\Gamma)$-convex (i.e., holomorphically convex in $\Gamma$), hence also $\Oscr(X)$-convex 
as well as polynomially convex in $\C^N\times\C^n$ since $X$ is Runge in $\C^N\times\C^n$. 
(See H\"ormander \cite{Hormander1990} and Stout \cite{Stout2007} for results on holomorphic convexity.)  
By \cite[Lemma 6.5, p. 111]{Forstneric1999JGEA} the compact set $(L\times K)\cup \Gamma_L$ is 
$\Oscr(X)$-convex and hence polynomially convex, so it has a basis of Runge Stein neighbourhoods.

Let $\pi:\C^N\times\C^n\to \C^N$ denote the projection on the first factor. 
Consider the injective $\pi$-fibre preserving holomorphic map $\Phi=(\Id,\phi)$ of the form \eqref{eq:Phi} 
on a small Runge Stein neighbourhood $\Omega=\Omega'\cup \Omega''$ 
of $(L\times K)\cup \Gamma_L$ in $\C^N \times \C^n$ which equals the identity map on a neighbourhood 
$\Omega'$ of $L\times K$ and whose second component equals
\[
	\phi(\zeta,z) = f(\zeta) + \frac12 (z-f(\zeta)) = \frac12 f(\zeta) + \frac12 z 
\]
for $(\zeta,z)$ in a neighbourhood $\Omega''$ of the graph $\Gamma_L$ in \eqref{eq:GammaL}. 
Thus, $\phi(\zeta,\cdotp)$ is a contraction by the factor $1/2$ around the point $f(\zeta)\in\C^n$ for every 
$\zeta\in L$. For a suitable choice of the neighbourhood $\Omega''$ of $\Gamma_L$ the map $\phi=\phi_{1/2}$ 
is connected to $\phi_0(\zeta,z)=z$ by the isotopy 
\[
	\phi_t(\zeta,z)  =  t f(\zeta) + (1-t)z \ \ \text{for}\ 0\le t\le \frac12.
\]
On $\Omega'$ we take the constant isotopy $\phi_t(\zeta,z)=\phi_0(\zeta,z)=z$ for all $t$. 
Clearly, the trace of the isotopy $\Phi_t=(\Id,\phi_t)$ for $t\in[0,1/2]$ 
consists of Runge domains $\Phi_t(\Omega)\subset \Omega$. 
By the parametric version of Theorem \ref{th:interpolationY} (see Remark \ref{rem:parametric}) 
we can approximate $\Phi$ as closely as desired on a Runge neighbourhood of $(L\times K)\cup \Gamma_L$
by a holomorphic map 
\[
	\Psi:V\times \C^n\to V\times \C^n,\quad  \Psi(\zeta,z)=(\zeta, \psi(\zeta,z)),
\]
where $V\subset U$ is a neighbourhood of $L$, such that for every $\zeta\in V$ we have that 
\begin{itemize}
\item $\psi(\zeta,\cdotp)\in \Aut(\C^n)$, 
\item $\psi(\zeta,z)=z$ for every $z=(z',0)\in \C^{n-1}\times\{0\}$, and
\item $\psi(\zeta,f(\zeta)) = f(\zeta)$. 
\end{itemize}
Choose a pair of constants $a,b\in \R$ such that 
\[
	0<a<1/2<b<1\quad \text{and}\quad  b^2<a. 
\]
If the approximation of $\phi$ by $\psi$ is close enough then the estimate 
\begin{equation}\label{eq:attracting}
	a|z-f(\zeta)| \le |\psi(\zeta,z)-f(\zeta)| \le b|z-f(\zeta)| 
\end{equation}
holds in a neighbourhood of the graph $\Gamma_L$ in \eqref{eq:GammaL}. At the same time, we can ensure
that $\psi$ is arbitrarily close to the map $(\zeta,z)\mapsto z$ on a neighbourhood of $L\times K$.

It is obvious that this result gives a sequence of holomorphic maps 
$\psi_k$ of the same kind as $\psi$ for $k=1,2,\ldots$ such that the estimate \eqref{eq:attracting} holds for
all of them on the same neighbourhood of $\Gamma_L$, and the sequence $\psi_k$ converges to 
the map $(\zeta,z)\mapsto  z$ on a neighbourhood of $L\times K$ as $k\to\infty$. 
Consider the sequence of automorphisms
\[
	\theta_k(\zeta,\cdotp) = 
	\psi_k(\zeta,\cdotp)\circ \psi_{k-1}(\zeta,\cdotp)\circ\cdots\circ \psi_{1}(\zeta,\cdotp) \in \Aut(\C^n)
\]
for $k\in\N$ and all $\zeta$ in a neighbourhood of $L$. 
Due to the condition $b^2<a$ in the estimate \eqref{eq:attracting}, which holds 
for all $k\in\N$, the attracting basin $B_\zeta\subset \C^n$ of the sequence $\theta_k$ at the fixed point 
$f(\zeta)$ is biholomorphic to $\C^n$ (see \cite{Wold2005}). If the convergence of the sequence $\psi_k$ to the map 
$(\zeta,z)\mapsto  z$ is fast enough on a neighbourhood of $L\times K$, which can be arranged by our construction, 
then no point of $K$ escapes a given neighbourhood of $K$, and hence none of the basins $B_\zeta$ intersect $K$. 
Furthermore, the condition $\psi_k(\zeta,(z',0))=(z',0)$ for all $\zeta$ near $L$, 
$z'\in \C^{n-1}$, and $k\in \N$ ensures that the basin $B_\zeta$ does not intersect the hyperplane 
$\C^{n-1}\times \{0\}$. By the general argument concerning attracting basins (see e.g.\  \cite{Wold2005}), 
this gives a holomorphic map $F:V\times\C^n\to \C^n$ such that the image $B_\zeta$ 
of $F(\zeta,\cdotp)$ is a Fatou--Bieberbach domain  in $(\C^{n-1}\times\C^*) \setminus K$ centred 
at $f(\zeta)$ for every $\zeta\in V$. 
\end{proof}

%
%
%   PROOF OF THE MAIN THEOREM
%
%
\section{Proof of Theorem \ref{th:main}} \label{sec:proofmain}

We now show how Theorem \ref{th:MRL2020Theorem1.1} 
%on the existence of holomorphically varying families of Fatou--Bieberbach domains in the complement of the union of a hyperplane and a polynomially convex set in $\C^n$ 
implies Theorem \ref{th:main}, and hence also Theorem \ref{th:mainbis}. 
We shall use the following characterization of Oka manifolds due to Kusakabe \cite[Theorem 1.3]{Kusakabe2021IUMJ}. 

\begin{theorem}\label{th:Ell1} 
A complex manifold $Y$ is an Oka manifold if and only if for every compact convex set $L\subset \C^N$,
integer $N\in\N$, open set $U\subset \C^N$ containing $L$, and holomorphic map $f:U\to Y$ 
there are an open set $V$ with $K\subset V\subset U$ and a holomorphic map $F:V\times \C^n\to Y$
for some $n\ge \dim Y$ such that $F(\cdotp,0)=f$ and 
\[
	\frac{\di}{\di z}\Big|_{z=0}F(\zeta,z):\C^n\to T_{f(\zeta)} Y \ \ \text{is surjective for every}\ \ \zeta \in V.
\] 
\end{theorem}

Such a map $F$ is called a {\em dominating holomorphic spray} with the core $f=F(\cdotp,0)$.

Note that the map $F$ in Theorem \ref{th:MRL2020Theorem1.1} is a dominating spray with a given core $f$.
Hence, the following is an immediate corollary to Theorems \ref{th:MRL2020Theorem1.1} and \ref{th:Ell1}.

%
%  	COROLLARY: OKA DOMAIN
%
\begin{corollary}\label{cor:Oka}
If $K$ is a compact polynomially convex set in $\C^n$ for some $n>1$ then 
$(\C^{n-1}\times \C^*)\setminus K$ is an Oka manifold. 
\end{corollary} 

By using this corollary we infer the following.

\begin{proposition}\label{prop:Oka}
Under the assumptions of Theorem \ref{th:main}, the domain $\C^n\setminus (E\cup\Lambda)$ is Oka. 
\end{proposition}

\begin{proof}
If $\Lambda = H$, then since $\overline E\cap \Lambda=\varnothing$, it follows that $E$ 
is compact and hence $\C^n\setminus E$ is Oka by \cite[Theorem 1.2 and Corollary 1.3]{Kusakabe2020complements} 
(see also \cite{ForstnericWold2020MRL}). 
Assume now that $\Lambda\ne H$ and $K=\overline E$ is a compact polynomially convex set in 
$\CP^n\setminus \Lambda\cong \C^n$. Choose affine complex coordinates 
$z=(z_1,\ldots,z_n)$ on $\CP^n\setminus \Lambda$ such that $H\setminus \Lambda=\{z_n=0\}$. 
Then, 
\[
	\C^n\setminus (E\cup\Lambda) = \CP^n\setminus (E\cup H\cup\Lambda) = 
	(\CP^n\setminus\Lambda) \setminus (H\cup K)=\{(z',z_n) : z_n\ne 0\}\setminus K.
\]
We are now in the situation of Corollary \ref{cor:Oka}, which gives the desired conclusion.
\end{proof}

We also recall the following result; see \cite[Theorem 5.1]{Forstneric2022Oka} and its proof.

\begin{theorem}\label{th:Oka5.1}
Let $K$ be a compact subset of $\CP^n$ for $n>1$. If there is a hyperplane
$\Lambda\subset\CP^n\setminus K$ such that $K$ is polynomially convex in 
$\CP^n\setminus\Lambda\cong\C^n$, then 
$K$ is polynomially convex in $\CP^n\setminus \Lambda'$ for every hyperplane 
$\Lambda'\subset \CP^n\setminus K$ which is connected to $\Lambda$ by a path of hyperplanes
in $\CP^n\setminus K$, and $\CP^n\setminus K$ is Oka.
\end{theorem}

\begin{proof}[Proof of Theorem \ref{th:main}] 
By Theorem \ref{th:Oka5.1} applied to the compact set $K=\overline E\subset\CP^n$ 
there are hyperplanes $\Lambda_0=\Lambda,\Lambda_1,\ldots, \Lambda_n$ in 
$\CP^n\setminus \overline E$ close to $\Lambda$ such that $\bigcap_{i=0}^n\Lambda_i=\varnothing$
and $\overline E$ is polynomially convex in $\CP^n\setminus \Lambda_i$ for $i=0,1,\ldots,n$.
Let $\C^n=\CP^n\setminus H$. By Proposition \ref{prop:Oka} the domain $\C^n\setminus (E\cup\Lambda_i)$ 
is Oka for every $i=0,\ldots,n$. Note that 
\[
	\C^n\setminus E =\bigcup_{i=0}^n \, \C^n\setminus (E\cup\Lambda_i). 
\]
Since every Oka domain $\C^n\setminus (E\cup\Lambda_i)=(\C^n\setminus E)\setminus \Lambda_i$ 
is Zariski open in $\C^n\setminus E$, the localization theorem for Oka manifolds \cite[Theorem 1.4]{Kusakabe2021IUMJ} 
(see also \cite[Theorem 3.6]{Forstneric2022Oka}) implies that $\C^n\setminus E$ is an Oka domain.
\end{proof}  

%
%
%	PROJECTIVELY CONVEX SETS
%
%
\section{Projectively convex sets}\label{sec:projconvex}

There are several notions of convexity for subsets of complex projective spaces; we refer to 
\cite{AnderssonPassareSigurdsson2004}. We shall be interested in the following one.

\begin{definition}\label{def:projconvex}
Let $K$ be a compact set in $\CP^n$.
\begin{enumerate} [\rm (i)] 
\item The set $K$ is weakly projectively convex if for every point $p\in \CP^n\setminus K$ there is a projective 
hyperplane $\Lambda\subset\CP^n$ with $p\in \Lambda$ and $\Lambda\cap K=\varnothing$.
\vspace{1mm}
\item The set $K$ is projectively convex if it is weakly projectively convex and the space $\Hcal_K$ 
of all projective hyperplanes $\Lambda\subset \CP^n$ with $K\cap\Lambda=\varnothing$ is connected.
\end{enumerate}
\end{definition}

Clearly, both notions are invariant under automorphisms of $\CP^n$. 
Our interest in projective convexity is that it gives a purely geometric sufficient condition for validity 
of Theorem \ref{th:main}. This gives many new examples of Oka domains in Euclidean and projective spaces.

%
%  PROJECTIVELY CONVEX SETS
%
\begin{theorem}\label{th:projconvex}
If $K\subsetneq\CP^n$ for $n>1$ is a compact projectively convex set, then the following assertions hold.
\begin{enumerate}[\rm (a)]
\item For every projective hyperplane $\Lambda\subset \CP^n$ with $\Lambda\cap K=\varnothing$ 
the set $K$ is polynomially convex in $\CP^n\setminus \Lambda\cong \C^n$.
\item $\CP^n\setminus K$ is an Oka manifold.
\item For every projective hyperplane $H\subset\CP^n$ the set $\CP^n\setminus (H\cup K)$ is an Oka manifold.
\end{enumerate}
\end{theorem}

\begin{proof}
To prove (a) we use the argument from \cite[proof of Theorem 5.1]{Forstneric2022Oka}.
Denote by $\Hcal_K$ the open set of hyperplanes $\Lambda\subset \CP^n$ with 
$K\cap\Lambda=\varnothing$. Note that $\Hcal_K$ is connected since $K$ is assumed to be projectively convex.
Fix  $\Lambda_0\in \Hcal_K$. Given a point $p\in \CP^n\setminus (K\cup\Lambda_0)$,
there is a path $\Lambda_t\in \Hcal_K$ for $t\in[0,1]$ connecting 
$\Lambda_0$ to a hyperplane $\Lambda_1$ with $p\in \Lambda_1$. We may assume that 
$\Lambda_t\ne \Lambda_0$ for $t\in (0,1]$. Note that 
$K\subset \CP^n\setminus \Lambda_0 = \C^n$, and $\Sigma_t:=\Lambda_t\setminus \Lambda_0\subset\C^n$ 
for $t\in  (0,1]$ is a path of affine complex hyperplanes in $\C^n\setminus K$ such that 
$p\in \Sigma_1$ and $\Sigma_t$ diverges to infinity as $t\to 0$. 
By Oka's criterion (see Oka \cite{Oka1937} and Stout \cite[Theorem 2.1.3]{Stout2007})
this means that $p$ does not belong to the polynomial hull of $K$ in $\CP^n\setminus \Lambda_0$. 
Since this holds for every point $p\in \CP^n\setminus (K\cup\Lambda_0)$, we conclude that $K$ is 
polynomially convex in $\CP^n\setminus \Lambda_0$. This proves (a).

To prove (b), choose hyperplanes $\Lambda_0,\Lambda_1,\ldots,\Lambda_n \in \Hcal_K$ such that 
$\bigcap_{i=0}^n \Lambda_i=\varnothing$. Then, 
\[
	\CP^n\setminus K =\bigcup_{i=0}^n \CP^n\setminus (K\cup \Lambda_i).
\]
By part (a), $K$ is polynomially convex in $\CP^n\setminus \Lambda_i\cong \C^n$ for every $i=0,\ldots, n$, 
and hence the domain $\CP^n\setminus (K\cup\Lambda_i)=(\CP^n\setminus \Lambda_i)\setminus K$ is Oka by  
\cite[Theorem 1.2 and Corollary 1.3]{Kusakabe2020complements}. 
This gives a covering of $\CP^n\setminus K $ by Zariski open Oka domains, so 
$\CP^n\setminus K$ is Oka by the localization theorem \cite[Theorem 1.4]{Kusakabe2021IUMJ}. 

Part (c) follows from (a) and Theorem \ref{th:mainbis}, which is equivalent to Theorem \ref{th:main} proved
in the previous section.
\end{proof} 

%
%   EXAMPLE
%
\begin{example}\label{ex:tubearoundhyperplane}
In $\C^n$ with coordinates $(z',z_n)$ we consider a domain of the form
\[
	\Omega=\{(z',z_n)\in\C^n : |z_n|^2 < c(1+|z'|^2)\},\quad c>0.
\]
Let $H$ denote the hyperplane at infinity and $\Lambda=\{z_n=0\}$. In suitable affine coordinates
$w=(w',w_n)$ on $\CP^n\setminus \Lambda = (\C^n\cup H)\setminus \Lambda\cong\C^n$ in which 
$\Lambda$ is the hyperplane at infinity and $H=\{w_n=0\}$, the domain $\Omega$ is the complement of 
$H\cup \overline{\B}$ where $\B$ is a ball centred at the origin. Hence, $\Omega$ is an Oka domain
by part (c) of Theorem \ref{th:projconvex}.
\end{example}

%
%   OKA'S CRITERION FOR POLYNOMIAL CONVEXITY
%
\begin{remark}
The use of hyperplanes in the proof of Theorem \ref{th:projconvex}
can be replaced by more general projective hypersurfaces
to give the following criterion for validity of Theorem \ref{th:Oka5.1}. 
Let $K$ be a compact subset of $\CP^n$ and $\Lambda\subset\CP^n\setminus K$
be a hyperplane. Assume that for every $p\in \CP^n\setminus (K\cup\Lambda)$ there is
a continuous $1$-parameter family of compact hypersurfaces $A_t\subset \CP^n\setminus K$
$(t\in [0,1])$ such that $p\in A_0$ and $A_t$ converges to $\Lambda$ as $t\to 1$ in the sense
that for every neighbourhood $U\subset\CP^n$ of $\Lambda$ there is a 
$c\in(0,1)$ such that $A_t\subset U$ for all $t\in (c,1)$. Then $K$ is polynomially convex in 
$\CP^n\setminus \Lambda$, and hence Theorem \ref{th:Oka5.1} holds.
This is seen by the argument in the proof of Theorem \ref{th:projconvex}, 
using Oka's criterion for polynomial convexity.
\end{remark}

%
%  GEOMETRIC CHARACTERIZATION OF CLOSED SETS WITH PROJECTIVELY CONVEX CLOSURES
%
We now give a geometric characterization of closed sets in $\C^n$ having projectively convex closures. 
Assume that $\Lambda$ is a complex affine subspace of dimension $k\in \{1,\ldots,n-1\}$ in $\C^n$ 
and $p\in \Lambda$. In suitable affine complex coordinates $z=(z',z'')\in \C^k\times  \C^{n-k}$ 
we have that $p=0$ and $\Lambda=\{z''=0\}$. Given $c>0$ we define
\begin{equation}\label{def:cone}
	C(\Lambda,p,c)=\{(z',z'')\in\C^n: |z''|\le c|z'|\}.
\end{equation}
This is a closed cone with the axis $\Lambda$ and vertex $p$. The analogous definition 
makes sense for real affine subspaces of $\R^n$.

The proof of the following elementary observation is left to the reader. % (cf.\ Remark \ref{rem:conditionE}).
 
\begin{lemma}\label{lem:stable}
Let $E$ be a closed set in $\C^n$ and $\Lambda$ be a complex affine subspace of $\C^n$.
Consider $\C^n$ as a domain in $\CP^n$ and set $H=\CP^n\setminus \C^n$.
The following are equivalent.
\begin{enumerate}[\rm (i)]
\item There is a point $p\in \Lambda$ and a number $c>0$ such that $C(\Lambda,p,c) \cap E$ is compact.
\item The projective closure of $\Lambda$ does not intersect $\overline E$ along $H$, i.e., 
$\overline E\cap \overline \Lambda\cap H=\varnothing$.
\end{enumerate}
We shall say that $\Lambda$ is {\em $E$-stable} if these equivalent conditions hold.
\end{lemma}

%
%   REMARKS ON STABILITY
%
\begin{remark}\label{rem:stable}
Let $E$ and $\Lambda$ be as in Lemma \ref{lem:stable}.
\begin{enumerate}[\rm (a)]
\item If the condition in Lemma \ref{lem:stable} (i)  holds for a point $p_0\in\Lambda$, then it holds for every 
point $p\in \Lambda$ (with a constant $c>0$ depending on $p$). 
\item If $\Lambda$ is $E$-stable then so is any parallel translate $\Lambda'$ of $\Lambda$. In fact,
$\overline \Lambda'\cap H=\overline \Lambda\cap H$. 
\item The space of $E$-stable $k$-dimensional affine subspaces $\Lambda\subset \C^n\setminus E$ is open.
\end{enumerate}
\end{remark}

%
%   OKA COMPLEMENTS OF (n-1)-CONVEX SETS
%
\begin{proposition}\label{prop:hyperconvex}
Let $E$ be a closed subset of $\C^n$. Then the closure $\overline E\subset\CP^n$ 
is projectively convex if and only if the following three conditions hold.
\begin{enumerate}[\rm (i)]
\item Every point $p\in \C^n\setminus E$ lies in an $E$-stable hyperplane $\Lambda\subset\C^n$
such that $E\cap\Lambda=\varnothing$. 
\item Every $E$-stable affine complex line $L\subset \C^n$ has a
parallel translate contained in an $E$-stable complex hyperplane $\Lambda \subset \C^n\setminus E$. 
\item The space of $E$-stable affine complex hyperplanes in $\C^n\setminus E$ is connected.
\end{enumerate}
\end{proposition}

\begin{proof}
Let $H=\CP^n\setminus \C^n$. In view of Lemma \ref{lem:stable}, condition (i) means that 
$\C^n \setminus E$ is a union of affine hyperplanes whose projective closures do not intersect
$\overline E$. Condition (ii) means that for every point $p\in H\setminus \overline E$ there exists
an affine hyperplane $\Lambda\subset \C^n\setminus E$ such that 
$p\in \overline \Lambda\subset \CP^n\setminus \overline E$. Hence, (i) and (ii) together
are equivalent to $\CP^n\setminus \overline E$ being the union of projective hyperplanes,  
i.e., weakly projectively convex. Finally, condition (iii) means that the space of projective hyperplanes lying 
in $\CP^n\setminus \overline E$ is connected. Hence, the three conditions together are equivalent to 
$\overline E$ being projectively convex.
\end{proof}

The following is an immediate corollary to Theorem \ref{th:main} and Proposition \ref{prop:hyperconvex}. 

\begin{corollary}\label{cor:projconvex}
If a closed set $E$ in $\C^n$ for $n>1$ satisfies the hypotheses in Proposition \ref{prop:hyperconvex}
then $\C^n\setminus E$ is an Oka domain. 
\end{corollary}

As an application of this result, we show the following.

\begin{proposition}\label{prop:C2minusR}
The complement of an affine real line in $\C^2$ is Oka. Furthermore, the complement of a closed disc-tube
around such a line is Oka.
\end{proposition}

\begin{proof}
Choose coordinates $z_1=x_1+\imath y_1,\ z_2=x_2+\imath y_2$ on $\C^2$ so that the 
line in question is the $x_1$-axis and the tube is $E=\{y_1^2+|z_2|^2<c\}$ for $c>0$. The intersection of
$\overline E\subset\CP^2$ with the line at infinity $H=\CP^2\setminus \C^2$ is the single point $p$ determined by 
the $z_1$-axis. Any other point $q\in H\setminus \{p\}$ determines a complex line which is linearly independent 
of the $z_1$-axis. Clearly, such a line is $E$-stable and a parallel translate of it avoids the $x_1$-axis 
by dimension reasons. It is also possible to avoid the tube $E$ around the $x_1$-axis. 
It is elementary to see that the space of affine complex lines avoiding $E$ and linearly independent 
of the $z_1$-axis is connected. Hence, the conclusion follows from Corollary \ref{cor:projconvex}. 
\end{proof}

%
%   C^n\R^k
%
\begin{remark}\label{rem:CnRk}
It was shown by Kusakabe \cite[Corollary 1.7]{Kusakabe2020complements} that if 
$(n,k)$ is a pair of integers $1\le k\le n$ with $n\ge 3$ and $(n,k)\ne (3,3)$ 
then for any closed set $E$ in $\R^k\subset \C^n$ the complement $\C^n\setminus E$ is Oka;
in particular, $\C^n\setminus \R^k$ is Oka for these pairs of values $(n,k)$. 
To prove this, Kusakabe used his theorem (see \cite[Theorem 1.6]{Kusakabe2020complements}) 
saying that if $E$ is a closed, possibly unbounded polynomially convex subset of 
$\C^n=\C^{n-2}\times \C^2$ which is contained in a set of the form 
$\{(z',z''):|z''|\le c(1+|z'|)\}$ with respect to some holomorphic coordinates 
$z=(z',z'')$ on $\C^{n-2}\times \C^2$ and $c>0$, then $\C^n\setminus E$ is Oka.

This approach does not work for the exceptional cases $(n,k) \in \{(2,1),(2,2),(3,3)\}$. 
Proposition \ref{prop:C2minusR} extends Kusakabe's result to the case $n=2$, $k=1$ 
by a completely different method. The remaining two cases $(2,2)$ and $(3,3)$ are not amenable 
to this method either, so they remain an open problem.

Let us look more closely at the domain $\C^2\setminus \R^2$. The closure 
of $\R^2$ in $\CP^2$ intersects the line at infinity in a circle $S$ 
which divides $\CP^1$ in two disc components. The points of $S$ correspond to those complex lines in
$\C^2$ which intersect $\R^2\subset\C^2$ along a real line; 
these are  the complex lines which fail to be transverse to $\R^2$. 
A point $p\in \CP^1\setminus S$ determines a complex line $\Lambda\subset\C^2$ 
intersecting $\R^2$ transversely, and hence every translate of $\Lambda$ intersects $\R^2$. 
Thus, the set $E=\R^2\subset \C^2$ does not satisfy the conditions in Proposition \ref{prop:hyperconvex}. 

A similar analysis can be made for $\C^3\setminus \R^3$.
\end{remark}

%
%
%    STRICTLY CONCAVE DOMAINS IN C^2 ARE OKA
%
%
\section{Convex domains in $\C^n$ with Oka complements}\label{sec:convex}

In this section we prove Theorem \ref{th:strictlyconvex} from the introduction.
We will show that for every closed convex set $E\subset \C^n$ satisfying the conditions
of that theorem, its projective closure $\overline E\subset\CP^n$ is projectively convex
(see Def.\ \ref{def:projconvex}), so the result will follow from Theorem \ref{th:projconvex}.

Given a domain $E\subset \C^n$ with $\Cscr^1$ boundary and a point 
$p\in bE$, we denote by $T^\C_p bE$ the maximal complex subspace (a complex hyperplane)
in the real tangent space $T_p bE$. (Both tangent spaces are considered 
as affine spaces passing through the point $p$.)  
Recall that a real affine subspace $p\in \Lambda\subset\C^n$ is said to be 
supporting for $E$ at $p\in bE$ if $\Lambda\cap E\subset bE$. If $E$ is convex and 
$bE$ is of class $\Cscr^1$ then this holds if and only if $\Lambda\subset T_p bE$, 
and if $\Lambda$ is complex then it holds if and only if $\Lambda\subset T_p^\C bE$.

The notion of an $E$-stable affine subspace was introduced in Lemma \ref{lem:stable}.

%
%  STABILITY AND HALFLINES
%
\begin{lemma}\label{lem:convex-stable}
Let $E$ be a closed convex set in $\C^n$ and $p\in bE$. Assume that $\Lambda\subset\C^n$
is a supporting affine complex subspace for $E$ at $p$. Then $\Lambda$ is $E$-stable if and only if 
$E\cap \Lambda$ does not contain a real halfline. The analogous result holds in the real setting.
\end{lemma}

\begin{proof}
Choose affine coordinates $z=(z',z'')$ on $\C^n=\C^k\times \C^{n-k}$ with $k=\dim \Lambda$ such that 
$p=0$ and $\Lambda=\{z''=0\}$.

If $E\cap \Lambda$ contains a real halfline $L$, then the terminal point of $L$ at infinity
lies in the projective closure $\overline E\subset\CP^n$ of $E$, 
so $\Lambda$ is not $E$-stable. Note also that, since $E$ is closed and convex,
the line segments $l_q$ connecting the given initial point $p$ to points $q\in L$ 
moving to infinity converge to a halfline in $E\cap \Lambda$ with the finite endpoint $p$.

To prove the converse, assume that $\Lambda$ is not $E$-stable. Then the intersection of $E$ with the
closed cone $C(\Lambda,p,c)$ in \eqref{def:cone} is noncompact (and hence unbounded) for every $c>0$. 
Letting $c\to 0$ we obtain a sequence of unit vectors $v_j=(v'_j,v''_j)\in\C^n$ and numbers $t_j >0$ 
such that $t_j v_j\in E$ for all $j\in\N$, $\lim_{j\to\infty } t_j= +\infty$, and $\lim_{j\to\infty }|v''_j|/|v'_j|=0$.
By passing to a subsequence we may assume that the sequence $v'_j/|v'_j|$
converges to a unit vector $v'\in \C^k$. The line segments $l_j\subset E$ connecting 
$p=0$ to the point $t_j v_j\in E$ then converge to the halfline $\R_+ v' \in E\cap \Lambda$ terminating at $p$.
\end{proof}

Recall that affine subspaces $\Lambda$ and $V$ in $\R^n$ are said to be complementary  
if and only if $\dim \Lambda+\dim V=n$ and their intersection is a point. For vector subspaces,
this holds if and only if $\R^n$ is their direct sum $\Lambda\oplus V$. 

%
%	INTERSECTING A STRICTLY CONVEX SET
%
\begin{lemma}\label{lem:divide}
Let $E$ be a closed convex set in $\R^n$, and let $\Lambda \subset \R^n$ be an affine subspace 
such that $E\cap \Lambda$ is bounded. Then the following are equivalent.
\begin{enumerate}[\rm (i)]
\item Every parallel translate $\Lambda'$ of $\Lambda$ intersects $E$.
\item There is a vector subspace $V\subset \R^n$ complementary to $\Lambda$ such that
$E=E\cap\Lambda + V$.
\end{enumerate}
If these equivalent conditions fail then there is a translate $\Lambda'$ of $\Lambda$ which
is a supporting subspace for $E$ at a point $q\in bE\cap \Lambda'$.
\end{lemma}

Note that the set $E\cap\Lambda + V$ in part (ii) is a tube with basis $E\cap\Lambda$ and fibre $V$.

\begin{proof}[Proof of Lemma \ref{lem:divide}] 
The implication (ii)\ $\Rightarrow$\ (i) is obvious. 

Let us now prove that (i)\ $\Rightarrow$\ (ii).
We begin with the case when $\Lambda$ is a hyperplane.
Choose coordinates $(x_1,\ldots,x_n)$ on $\R^n$ such that $\Lambda=\{x_n=0\}$.
By the assumption, the closed convex set $E'=E\cap \Lambda$ is bounded.
Let $H^\pm =\{\pm \, x_n\ge 0\}$. The assumption that every translate of $\Lambda$ intersects 
the set $E^+:=E\cap H^+$ gives a sequence of unit vectors $v_j\in H^+$ and numbers $t_j>0$ with 
$\lim_{j\to\infty} t_j=+\infty$ such that $t_j v_j\in E^+$. By compactness
of the unit sphere we may pass to a subsequence and assume that $\lim_{j\to\infty} v_j=v$.
Fix a point $x'\in E'$. The line segments connecting $x'$ to the points $t_j v_j\in E_+$ converge as $j\to\infty$
to the halfline $x'+\R_+ v$. Since $E$ is closed and convex, this halfline belongs to $E^+$.
In particular, $v\notin \Lambda$ since $E\cap \Lambda$ is bounded. 
This shows that $E^+$ contains the tube $E'+\R_+ v$. Since $E^-:=E\cap H^-$ is unbounded as well, 
the analogous argument shows that it contains a tube $E'+\R_+ w$ for some unit vector $w\in H^-$. 
Note that $w=-v$, for otherwise the convex hull of the union of these tubes contains a point in $\Lambda\setminus E'$. 
Hence, $E$ contains the tube $E'+\R v$, and the same argument as above shows that $E=E'+\R v$.

The above argument also shows that if $E$ is not a tube of the form $E'+\R v$ then at least one of the sets
$E^\pm :=E\cap H^\pm$ is bounded, and hence there is a parallel translate of $\Lambda$ satisfying
the last statement in the lemma.

Consider now the general case. We may assume that $p=0$ and 
\[
	\Lambda=\{x\in \R^n:x_{k+1}=0,\ldots,x_n=0\}.
\]
For every $j\in \{k+1,\ldots,n\}$ let $V_j\subset\R^n$ denote the subspace of dimension $k+1$ spanned 
by the coordinate directions $1,\ldots,k$ and $j$. Then $\Lambda$ is a hyperplane in $V_j$, so the 
special case proved above gives $E\cap V_j=E'+\R\cdotp v_j$ for some unit vector $v_j\in V_j\setminus \Lambda$.
Due to convexity it follows that $E$ contains the tube $E'+V$ with $V=\span\{v_{k+1},\ldots,v_n\}$.
Thus, $V$ is an $(n-k)$-dimensional subspace of $\R^n$ complementary to $\Lambda$.
If $E$ contains a vector $w\in\R^n$ not in $E'+V$, then convex combinations of $w$
and vectors from $E'+V$ give points in $E\cap \Lambda$ which are not contained in $E'$, a contradiction. 
\end{proof}

%
%  TRANSLATING AN AFFINE SUBSPACE
%
\begin{corollary}\label{cor:support}  
If $E$ is as in Theorem \ref{th:strictlyconvex} and $\Lambda\subset \C^n$ is an affine complex subspace 
such that $E\cap \Lambda$ is bounded, then there is a parallel translate $\Lambda'$ of $\Lambda$ with 
$E\cap \Lambda'=\varnothing$, and also one satisfying $E\cap \Lambda'=\{p\}$ with $p\in bE$. 
Furthermore, $\Lambda$ is $E$-stable.
\end{corollary}

\begin{proof}
If no translate of $\Lambda$ avoids $E$ then by Lemma \ref{lem:divide} the set $E$ is a tube $E'+V$, where
$E'=E\cap \Lambda$ and $V\subset\C^n$ is a real subspace complementary to $\Lambda$. 
Hence, at every point $p\in bE$ the tangent hyperplane $T_p bE$ contains the affine real 
subspace $p+V$ of dimension at least two. Since $T^\C_p bE$ is a real hyperplane in $T_p bE$,
its intersection with $p+V$ contains a real line, contradicting the hypothesis in Theorem \ref{th:strictlyconvex}.
This shows that there is a translate $\Lambda'$ of $\Lambda$ avoiding $E$, and also one 
which is a supporting subspace for $E$ at a point $p\in bE$ (see Lemma  \ref{lem:divide}). 
In the latter case, $\Lambda'$ (being a complex affine subspace) is contained in $T^\C_p bE$. 
The assumption in Theorem \ref{th:strictlyconvex} that $E\cap T^\C_p bE$ contains no halfline
implies by Lemma \ref{lem:convex-stable} that $T^\C_p bE$ is $E$-stable. 
Hence, $\Lambda'\subset T^\C_p bE$ is also $E$-stable, and the same holds for $\Lambda$ 
since this property is translation invariant.
\end{proof}

%
%   PROOF OF THEOREM "STRICTLYCONVEX"
%
\begin{proof}[Proof of Theorem \ref{th:strictlyconvex}]
We claim that $\overline E\subset\CP^n$ is projectively convex, so the result will follow from Theorem \ref{th:projconvex}.
By Proposition \ref{prop:hyperconvex} we must verify the following conditions:
\begin{enumerate}[\rm (i)]
\item $\C^n\setminus E$ is a union of $E$-stable affine complex hyperplanes, 
\item every $E$-stable complex line $L\subset \C^n$ has a
parallel translate contained in an $E$-stable complex hyperplane $H \subset \C^n\setminus E$, and
\item the set of all $E$-stable affine complex hyperplanes in $\C^n\setminus E$ is connected. 
\end{enumerate}

{\em Proof of (i):} Choose a point $q\in \C^n\setminus E$. Let $p\in bE$ be the closest point to $q$ in $E$.
The tangent plane $T_p bE$ is then a supporting hyperplane for $E$ and is orthogonal to the real line through $p$ and $q$.
The affine complex tangent plane $T^\C_p bE$ is $E$-stable by Lemma \ref{lem:convex-stable}.
The parallel translate $H$ of $T^\C_p bE$ to the point $q$ is then contained in $\C^n\setminus E$, 
and it is $E$-stable since this property is translation invariant.

{\em Proof of (ii):} Let $L$ be an $E$-stable affine complex line. Then $E\cap L$ is bounded,
and Lemma \ref{lem:divide} shows that a parallel translate $L'$ of $L$ 
is tangent to $bE$ at some point $p\in bE$. Then, $L'\subset T^\C_p bE$, which is an $E$-stable
complex hyperplane by Lemma \ref{lem:convex-stable}.
Translating $T^\C_p bE$ away from $E$ gives an $E$-stable hyperplane $H\subset \C^n\setminus E$ 
containing a translate of $L$. 

{\em Proof of (iii):} We claim that a complex affine hyperplane $H$ in $\C^n$ is $E$-stable if and only
if it is parallel to the complex tangent space $T_p^\C bE$ for some point $p\in bE$. 
In one direction, Lemma \ref{lem:convex-stable} shows that for every $p\in bE$ the 
complex tangent space $T^\C_p bE$ is $E$-stable. Conversely, if $H$ is $E$-stable then $H\cap E$ is a bounded set,
and Corollary \ref{cor:support} shows that a parallel translate of $H$ is tangent to $bE$ at some point $p\in bE$,
so this translate equals $T_p^\C bE$. It is easily seen that if $E$ is as in the theorem then 
its boundary $bE$ is connected, so the above shows that the set of $E$-stable hyperplanes is connected. 

It remains to see that the set $\Ecal$ of $E$-stable complex affine hyperplanes in $\C^n$
which do not intersect $E$ is also connected. If a hyperplane $H\subset \C^n\setminus E$ is $E$-stable, then by
Corollary \ref{cor:support} we can translate $H$ within $\C^n\setminus E$ until it hits $bE$ for the first time
at some point $p\in bE$, and this new hyperplane $H'$ is then equal to $T_p^\C bE$ and 
$H'\cap E\subset bE$. This shows that the set of $E$-stable affine complex hyperplanes contained in 
$\C^n\setminus \mathring E$ is connected. Take an interior point $q\in \mathring E= E\setminus bE$; we may assume that 
$q=0$. Consider the decreasing family of closed convex domains $E_k=(1+1/k) E$, i.e., we dilate $E$
by the factor $1+1/k$. Note that $E\subset \mathring E_k$ for all $k$ and $E=\bigcap_{k=1}^\infty E_k$.
Clearly, every $E_k$ has the same properties as $E$ and the same set of stable hyperplanes.
From what has been shown above, the set $\Ecal_k$ of $E$-stable hyperplanes avoiding 
$\C^n\setminus \mathring E_k$ is connected. As $k\to\infty$, the domains $E_k$ shrink down to $E$,
and hence $\Ecal=\bigcup_{k=1}^\infty \Ecal_k$ is an increasing union of connected sets, so it is connected.
\end{proof}

%
%
%   SECTION: INTERSECTIONS OF STRICTLY CONVEX SETS
%
%
\section{Intersections of strictly convex sets}\label{sec:intersection}

It is well known that every closed convex set in a Euclidean space $\mathbb R^m$ is an intersection of halfspaces. 
In this section we characterise those closed sets that are intersections of strongly convex sets. 
A closed set in $\R^m$ is called {\em strongly convex} if it has $\Cscr^2$ boundary and the principal normal curvatures 
of the boundary are strictly positive at every point. Clearly a strongly convex domain is also strictly convex,
but the converse fails in general. 

The following result is used in the proof of Theorem \ref{th:convexnoline}.

\begin{theorem}\label{thm:int}
If $E$ is a closed convex set in $\R^m$ which does not contain an affine line, then 
$E=\bigcap_j E_j$ where every $E_j$ is a closed strongly convex set and $E_{j+1}\subset E_j$ for $j=1,2,\ldots$. 
\end{theorem}

Conversely, a convex set containing a line is not contained in any strictly convex set.

\begin{proof}
Recall that in \cite[Chapter 1]{AnderssonPassareSigurdsson2004} 
a set $K\subset\mathbb{RP}^m$ is called convex if $(i)$ $K$ does not contain a projective line, and 
$(ii)$ the intersection of any projective line with $K$ is connected.
(There are other notions of convexity in projective spaces, but this is the one relevant here.) 
The assumption that $E\subset\R^m$ is convex and does not contain any real line 
guarantees that its closure $K=\overline E$ in $\mathbb{RP}^m$ is convex. 

Fix a point $p\in\mathbb{RP}^m\setminus K$. By \cite[Theorem 1.3.11]{AnderssonPassareSigurdsson2004}
there is a hyperplane $H\subset \mathbb{RP}^m\setminus K$ with $p\in H$. 
If $p\in\mathbb R^m$ then the real hyperplane $H'=H\cap \mathbb R^m$ contains $p$ and a proper cone $Q$ around $H'$ 
avoids $E$. (In the terminology introduced in Lemma \ref{lem:stable}, $H'$ is $E$-stable.)  
We may assume that $p$ is the origin, $H'=\{x=(x_1,\ldots,x_m)\in\mathbb R^m: x_m=0\}$, and 
$E$ is contained in the halfspace $\{x_m>0\}$. Then there exists a smooth strongly convex function 
$f(x_1,\ldots,x_{m-1})$ (i.e., with positive definite Hessian at every point) such that $f(0)=0$ and the 
graph of $f$ is contained in a strictly smaller cone $Q'\subset Q$. For such $f$ and $\epsilon>0$ small enough the set 
\[
	E' = \{x\in\mathbb R^m: x_m\geq f(x_1,...,x_{m-1})+\epsilon\}
\]
is strongy convex with $E\subset E'$ and $0\in\mathbb R^m\setminus E'$.  Hence, the function
\[
	\rho(x_1,\ldots,x_m)=\E^{f(x_1,...,x_{m-1})-x_m} - 1
\]
is strongly convex on $\R^m$, $\rho(0)=0$, and $\rho\le \E^{-\epsilon}-1<0$ on $E$. 

By scaling, it follows that for every compact set $K\subset\mathbb R^m\setminus E$ there exist 
strongly convex functions $\rho_1,...,\rho_k$ such that $E\subset\{x\in\mathbb R^n:\rho_j(x)\leq -1, j=1,..,k\}$
such that for every $x\in K$ we have that $\rho_j(x)>0$ for some $j\in\{1,...,k\}$. 
Exhausting $\mathbb R^m\setminus E$ by compact sets, 
it follows that there exist  strongly convex functions $\{\rho_j\}_{j=1}^\infty$ such that 
$E=\bigcap_{j=1}^\infty \{x\in\mathbb R^m:\rho_j(x)\leq 0\}$ and $\rho_j(x)\leq -1$ for $x\in E$ for all $j$. 

It remains to find a decreasing sequence of smoothly bounded strongly convex sets 
$E_1\supset E_2\supset\cdots \supset \bigcap_{k=1}^\infty E_k=E$.
We begin by taking $\tau_1=\rho_1$ and $E_1=E'_1=\{\tau_1\le 0\}$. 
To get $E_2$, we take the convex set $E_1\cap E'_2=\{\max\{\tau_1,\rho_2\}\le 0\}$ and smoothen the corners.
This is done by a {\em regularized maximum function} defined as follows (see \cite[p.\ 69]{Forstneric2017E}). 
Given a number $\delta>0$ we select a nonnegative smooth even function $\xi\ge 0$ on $\R$ with support in 
$[-\frac{\delta}{2},\frac{\delta}{2}]$ such that $\int \xi(t)dt=1$, and we set for $(u_1,u_2)\in\R^2$:
\begin{equation}\label{eq:rmax}
 	   \rmax\{u_1,u_2\} = \int_{\R^2}\max\{t_1+u_1,t_2+u_2\}\, \xi(t_1) \xi(t_2) dt_1\, dt_2.
\end{equation}
It is easily verified that the function $\rmax$ is smooth, increasing in every variable, and 
convex jointly in both variables. Hence, if $u_1(x)$ and $u_2(x)$ are (strongly) convex functions then 
$\rmax\{u_1(x),u_2(x)\}$ is also (strongly) convex. Moreover, we have that 
\[
	\max\{u_1,u_2\} \le \rmax\{u_1,u_2\}\le  \max\{u_1,u_2\}+\delta\ \ \text{for all}\ (u_1,u_2)\in\R^2
\]
and 
\[
    \rmax\{u_1,u_2\} = \begin{cases}    u_1, & \text{if}\ u_2 \le u_1-\delta,\\
    							     u_2, & \text{if}\ u_1 \le u_2-\delta. 
				   \end{cases}
\]
In a similar one defines $\rmax$ for any finite number of variables.

This shows that for a suitable choice of $\xi$ in the definition of $\rmax$ the function 
\[
	\tau_2:=\rmax\{\tau_1,\rho_2\}:\R^m\to\R
\] 
is smooth strongly convex and satisfies $\tau_2(a_1)>0,\ \tau_2(a_2)>0$, and $\tau_2\le -c_2<0$ 
on $E$ for some $c_2>0$. The set $E_2=\{\tau_2\le 0\}$ is strongly convex with smooth boundary, and 
it satisfies $E\subset \mathring E_2\subset E_2\subset E_1\cap E'_2$ and $a_2\notin E_2$.
 
%Clearly this process can be continued inductively. 
Assume inductively that we have constructed strongly convex 
functions $\tau_j:\R^m\to\R$ for $j=1,\ldots,k$ such that the sets $E_j=\{\tau_j\le 0\}$ 
satisfy $E\subset E_k\subset E_{k-1}\cdots\subset E_1$, 
$\tau_j\le c_j<0$ on $E$ for all $j=1,\ldots,k$, and $\tau_j(a_i)>0$ for $1\le i\le j$). We then take 
\[
	\tau_{k+1}=\rmax\{\tau_k,\rho_{k+1}\}
\] 
for a suitable choice of the weight function $\xi$ in $\rmax$ to obtain the next strongly
convex function $\tau_{k+1}$ and the next set $E_{k+1}=\{\tau_{k+1}\le 0\}$ in the sequence.
\end{proof}

%
%
%   ACKNOWLEDGMENTS
%
%
\subsection*{Acknowledgements}
The first named author is supported by research program P1-0291 and grants J1-9104 and N1-0237
from ARRS, Republic of Slovenia. The authors wish to thank Barbara Drinovec Drnov\v sek and Geir Dahl 
for discussions regarding convexity theory, and Yuta Kusakabe and Finnur L\'arusson for their 
helpful comments.

%%%%%%%%%%
%%%%%%%%%%
%%%%%%%%%%
%%%%%%%%%%   THE BIBLIOGRAPHY
%%%%%%%%%%
%%%%%%%%%%

%{\bibliographystyle{abbrv} \bibliography{references}} 
%\begin{comment}

%\end{comment}

\newpage

%%%%%%%%%%
%%%%%%%%%%
%%%%%%%%%%
%%%%%%%%%%   AFFILIATIONS
%%%%%%%%%%
%%%%%%%%%%

%\vspace*{5mm}
\noindent Franc Forstneri\v c

\noindent Faculty of Mathematics and Physics, University of Ljubljana, Jadranska 19, SI--1000 Ljubljana, Slovenia

\noindent 
Institute of Mathematics, Physics and Mechanics, Jadranska 19, SI--1000 Ljubljana, Slovenia.

\noindent e-mail: {\tt franc.forstneric@fmf.uni-lj.si}

\vspace*{5mm}
\noindent Erlend Forn\ae ss Wold

\noindent Department of Mathematics, University of Oslo, PO-BOX 1053, Blindern, 0316 Oslo. Norway. 

\noindent e-mail: {\tt erlendfw@math.uio.no}

\end{document}